\documentclass[11pt,a4paper]{article}
\usepackage{amsmath,amssymb,latexsym,graphics,epsfig}
\usepackage{hyperref}
\usepackage{color}
\usepackage{amsthm}
\usepackage{graphicx,url}
\usepackage{amsopn}
\usepackage{amssymb}
\usepackage[english]{babel}
\usepackage{soul}

\newtheorem{proposition}{Proposition}[section]
\newtheorem{theorem}{Theorem}[section]
\newtheorem{lemma}{Lemma}[section]

\newtheorem{definition}{Definition}
\newtheorem*{example*}{Example}

\newcommand\blfootnote[1]{%
	\begingroup
	\renewcommand\thefootnote{}\footnote{#1}%
	\addtocounter{footnote}{-1}%
	\endgroup
}

\DeclareMathOperator{\Circ}{circ}

\DeclareMathOperator{\spec}{sp}

\DeclareMathOperator{\Cay}{Cay}
\DeclareMathOperator{\diag}{diag}

\DeclareMathOperator{\Irep}{Irep}
\DeclareMathOperator{\rank}{rank}
\DeclareMathOperator{\Stab}{Stab}
\DeclareMathOperator{\Sym}{Sym}

\def\S{\mbox{\boldmath $S$}}

\def\Par{\pi}

\def\G{\Gamma}

\begin{document}
	
	\title{Spectra and eigenspaces from regular partitions\\ of Cayley (di)graphs of permutation groups
		\thanks{This research has been partially supported by 
			AGAUR from the Catalan Government under project 2017SGR1087 and by MICINN from the Spanish Government under project PGC2018-095471-B-I00. The research of the first author has also been supported by MICINN from the Spanish Government under project MTM2017-83271-R.}}
			
	\author{C. Dalf\'o\\
		{\small Departament de Matem\`atica, Universitat de Lleida} \\
		{\small Igualada (Barcelona), Catalonia} \\
				\vspace{.25cm}
		{\small {\tt{cristina.dalfo@matematica.udl.cat}}}\\
	M. A. Fiol\\
    	{\small Departament de Matem\`atiques, Universitat Polit\`ecnica de Catalunya} \\
    	{\small Barcelona Graduate School} \\
    	{\small Barcelona, Catalonia} \\
    	{\small {\tt{miguel.angel.fiol@upc.edu}}}
    }
	\date{}

	\maketitle
	
	\blfootnote{
		\begin{minipage}[l]{0.3\textwidth} \includegraphics[trim=10cm 6cm 10cm 5cm,clip,scale=0.15]{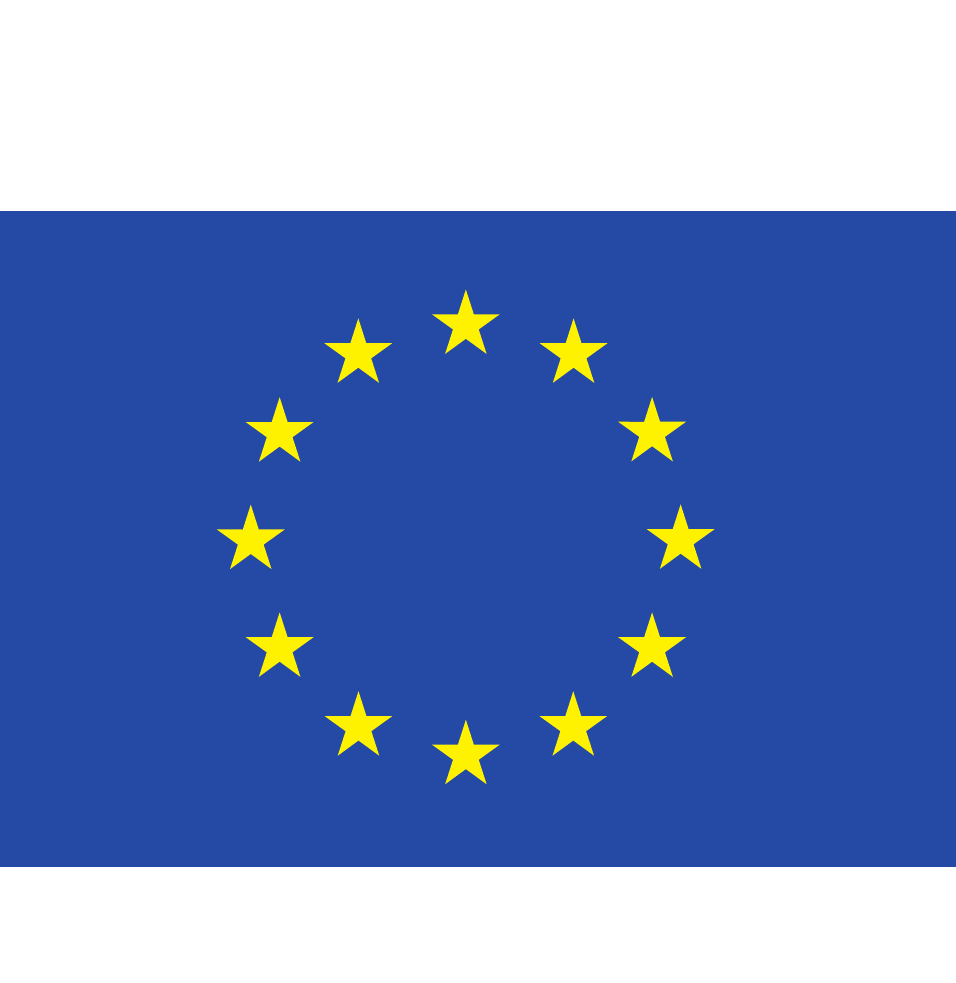} \end{minipage}  \hspace{-2cm} \begin{minipage}[l][1cm]{0.79\textwidth}
			The research of the first author has also received funding from the European Union's Horizon 2020 research and innovation programme under the Marie Sk\l{}odowska-Curie grant agreement No 734922.
	\end{minipage}}
	
\begin{abstract}
	In this paper, we present a method to obtain regular (or equitable) partitions of Cayley (di)graphs (that is, graphs, digraphs, or mixed graphs) of permutation groups on $n$ letters. We prove that every partition of the number $n$ gives rise to a regular partition of the Cayley graph.  By using representation theory, we also obtain the complete spectra and the eigenspaces of the corresponding quotient (di)graphs. More precisely, we provide a method to find all the eigenvalues and eigenvectors of such (di)graphs, 
	based on 
	their irreducible representations. As examples, we apply this method to the pancake graphs $P(n)$ and to a recent known family of mixed graphs $\G(d,n,r)$ (having edges with and without direction). As a byproduct, the existence of perfect codes in $P(n)$ allows us to give a lower bound for the multiplicity of its eigenvalue $-1$.
\end{abstract}
	

\section{Preliminaries}

In this paper, we study the eigenvalues and eigenvectors of Cayley (di)graphs Cay$(G,S)$ (in these, we include graphs, digraphs, and mixed graphs), where $G$ is a subgroup of the symmetric group $S_n=\Sym(n)$, and $S$ is the generating set given by some permutations $\pi_1,\pi_2,\ldots,\pi_k$.

Throughout this paper, $\Gamma=(V,E)$ denotes a digraph, which as said before can be a graph, digraph or mixed graph, with vertex set $V$ and arc set $E$. An arc from vertex $u$ to vertex $v$ is denoted by either $uv$ or $u\rightarrow v$. The set of vertices adjacent from a vertex $u\in V$ is denoted by $\Gamma^+(u)=\{v\in V: u\rightarrow v\}$. We allow {\em loops} (that is, arcs from a vertex to itself), and {\em multiple arcs}. A {\em digon} is a pair of opposite arcs, $uv$ and $vu$, forming an edge, and is denoted by $u\sim v$. So from now on, and without loss of generality, we refer to $\Gamma$ as a digraph, unless stated otherwise.
In particular if $\G$  contains both edges and arcs, it is usually referred to as a {\em mixed} (or {\em partially directed\/}) graph.
For more details, see the comprehensive survey of Miller and \v{S}ir\'a\v{n} \cite{ms13}.

If $\G$ has adjacency matrix $A$, its spectrum
$$
\spec \Gamma=\spec A =\{[\lambda_0]^{m_0},[\lambda_1]^{m_1},\ldots,[\lambda_d]^{m_d}\},
$$
 is constituted by the (possibly complex) distinct eigenvalues 
with the corresponding algebraic multiplicities $m_i$, for $i\in [n]=\{1,\dots,n\}$.


\subsection{Regular partitions and their spectra}

\label{sec:reg-part}
Let $\G=(V,E)$ be a digraph with adjacency matrix $A$. A partition $\Par=(V_1,\ldots,
V_m)$ of its vertex set $V$ is called {\em regular} (or {\em equitable})
whenever, for any $i,j=1,\ldots,m$, the {\em intersection numbers}
$b_{ij}(u)=|\G^+(u)\cap V_j|$, where $u\in V_i$, do not depend on the vertex $u$ but only on the subsets (usually called {\em classes} or {\em
cells}) $V_i$ and $V_j$. In this case, such numbers are simply written as $b_{ij}$, and the $m\times m$ matrix $B=(b_{ij})$ is referred to as the {\em quotient matrix} of $A$ with respect to $\Par$. This is also represented by the {\em quotient (weighted) digraph} $\pi(\G)$ (associated with the partition $\pi$),
with vertices representing the cells, and there is an arc with weight $b_{ij}$ from vertex
$V_i$ to vertex $V_j$ if and only if $b_{ij}\neq 0$.

The {\em characteristic matrix} of a partition $\Par$ is the $n\times m$ matrix
$\S=(s_{ui})$ whose $i$-th column is the characteristic vector of $V_i$, that is, $s_{ui}=1$ if $u\in V_i$, and $s_{ui}=0$ otherwise. In terms of this matrix, we have the following characterization of regular partitions and their spectra (see Godsil \cite{g93}).

\begin{lemma}[\cite{g93}]
	\label{g93}
	Let $\G=(V,E)$ be a digraph with adjacency matrix $A$, and vertex partition $\Par$
	with characteristic matrix $S$. 
\begin{itemize}
\item[$(i)$]
The partition  $\Par$ is regular if and only if there
exists an $m\times m$ matrix $C$ such that $SC=AS$. Moreover, $C=B$, the quotient matrix of $A$ with respect to $\Par$.
\item[$(ii)$]
If $\Par$ is regular and $x$ is an eigenvector of $B$, then $Sx$ is an eigenvector of $A$. Consequently, the spectrum of $\pi(\G)$ is contained in the spectrum of $\G$, that is, $\spec B\subseteq \spec A$.
\end{itemize}
\end{lemma}

\subsection{Lift digraphs and their spectra}
\label{sec:sp}
Given a group $G$ with generating set $S$, a
{\em voltage assignment} of the {\em base digraph} $\G$ is a mapping $\alpha:E\rightarrow S$. The pair $(\G,\alpha)$ is often called a {\em voltage digraph}. The {\em lifted digraph} (or, simply, {\em lift})
$\Gamma^\alpha$ is the digraph with vertex set  $V(\Gamma^\alpha)=V\times G$ and arc set $E(\Gamma^\alpha)=E\times G$, where there is an arc from the vertex $(u,g)$ to the vertex $(v,h)$ if and only if $uv\in E$ and 
$h=\alpha(uv)g$.
In this case, we 
	refer to a {\em regular} lift because of the mapping $\phi: \G^{\alpha}\rightarrow \G$ defined by erasing the second coordinate (that is, $\phi(u,g)=u$ and $\phi(a,g)=a$ for every $u\in N$ and $a\in E$) is a regular ($|G|$-fold) covering, in its usual meaning in algebraic topology (see, for instance, Gross and Tucker \cite{gt77}).

As a particular case of a lifted graph, notice that the Cayley digraph $\Cay(G,S)$ can be seen as a lift of the base digraph $\G$ consisting of a vertex with $|S|$ (directed) loops, each of them having assigned, through $\alpha$, an element of $S$.

	To the pair $(\Gamma,\alpha)$, we assign the $k\times k$ {\em base matrix} $B$, a square matrix whose rows and columns are indexed by the elements of the vertex set of $\Gamma$, and whose $uv$-th element $B_{u,v}$ is determined as follows: If $a_1,\ldots,a_j$ is the set of all the arcs of $\Gamma$ emanating from $u$ and terminating at $v$ (not excluding the case $u=v$), then
	\begin{equation}
	\label{eq:Balpha} B_{u,v}=\alpha(a_1)+\cdots + \alpha(a_j),
	\end{equation}
	the sum being an element of the complex group algebra $\mathbb{C}(G)$; otherwise, we let $B_{u,v}=0$.
	Given a   unitary irreducible representation of $G$, $\rho \in \Irep(G)$, of dimension $d_\rho$, let $\rho(B)$ be the $d_\rho k\times d_\rho k$ matrix obtained from $B$ by replacing every 
	entry $B_{u,v} \in \mathbb{C}(G)$ as in \eqref{eq:Balpha} by the $d_\rho\times d_\rho$ matrix 
	\begin{equation}
	\rho(B_{u,v}) =  \left\lbrace
	\begin{array}{cc}
	\rho(\alpha(a_1)) + \cdots + \rho(\alpha(a_j)) & \mbox{if } B_{u,v}\neq0,\\
	O & \mbox{otherwise},
	\end{array}
	\right.
	\label{eq:B}
	\end{equation}
	where $O$ is the all-zero $d_\rho\times d_\rho$ matrix.

The following results from \v{S}ir\'a\v{n} and the authors \cite{dfs19} (see also \cite{dfmrs17})
allow us to compute the spectrum of a (regular) lifted digraph from its associated matrix and the irreducible representations 
of its corresponding group. For more information on representation theory, see James and Liebeck \cite{JaLi} or Burrow \cite{b93}.

\begin{theorem}[\cite{dfs19}]
	\label{theo-sp}
Let $\G=(V,E)$ be a base digraph on $k$ vertices, with a voltage assignment $\alpha$ in a group $G$, with $|G|=n$. For every irreducible representation $\rho\in \Irep G$, let $\rho(B)$ be the complex matrix whose entries are given by \eqref{eq:B}.
Then,
$$
\spec \Gamma^{\alpha} = \bigcup_{\rho\in \Irep(G)}d_\rho\cdot\spec(\rho(B)).
$$
\end{theorem}

%
%

The result of Theorem \ref{theo-sp} can be generalized to deal with
 the so-called relative voltage assignments and (not necessarily regular) lifts, which are defined as follows.
 Let $\Gamma=(V,E)$ be the digraph considered above, $G$ a group, and $H$ a subgroup of $G$ of index $n$. Let $G/H$ denote the set of left cosets of $H$ in $G$. Furthermore, let $\beta: E\to G$ be a mapping  defined on every arc $a\in E$. In this context, one calls $\beta$ a {\em voltage assignment in $G$ relative to $H$}, or simply a {\em relative voltage assignment}. Then, the {\em relative lift} $\Gamma^{\beta}$ has vertex set $V^{\beta}= V\times G/H$ and arc set $E^{\beta}=E\times G/H$. Incidence in the lift is given as expected: If $a$ is an arc from a vertex $u$ to a vertex $v$ in $\Gamma$, then for every left coset $J\in G/H$ there is an arc $(a,J)$ from the vertex $(u,J)$ to the vertex $(v,\beta(a)J)$ in $\Gamma^{\beta}$.
Notice that a relative voltage assignment $\beta$ in a group $G$ with  subgroup $H$ is equivalent to a 
regular voltage assignment if and only if $H$ is a normal subgroup of $G$. In such a case, the relative lift $\G^{\beta}$ admits a description in terms of ordinary voltage assignment in the factor group $G/H$, with voltage $\beta(a)H$ assigned to an arc $a\in E$ with original relative voltage $\beta(a)$.
In this context, Pavl\'ikov\'a, \v{S}ir\'a\v{n}, and the authors \cite{dfss19} proved the following result, which generalizes Theorem \ref{theo-sp} for relative voltage assignments.

\begin{theorem}[\cite{dfss19}]
	\label{theo-sp2}
\label{theo-sp2}
Let $\Gamma$ be a base digraph of order $k$ and let $\beta$ be a voltage assignment on $\Gamma$ in a group $G$ relative to a subgroup $H$ of index $n$ in $G$. Given an irreducible representation $\rho\in \Irep(G)$, let us consider the matrix $\rho(H)= \sum_{h\in H}\rho(h)$. Then,
$$
\spec \Gamma^{\beta} = \bigcup_{\rho\in \Irep(G)}\rank(\rho(H))\cdot\spec(\rho(B)),
$$
where the union must be understood for all $\rho\in \Irep(G)$ such that $\rank(\rho(H))\neq 0$.
\end{theorem}

\subsection{The pancake graphs}
To illustrate our results, we use two families of Cayley graphs: The pancake graphs and a new family of mixed graphs introduced in \cite{Da19}, which can be seen as a generalization of both the pancake graphs and the cycle prefix digraphs. Let us first introduce the pancake graphs, together with some of their basic properties.

\begin{figure}[t]
	\begin{center}
		\includegraphics[width=14cm]{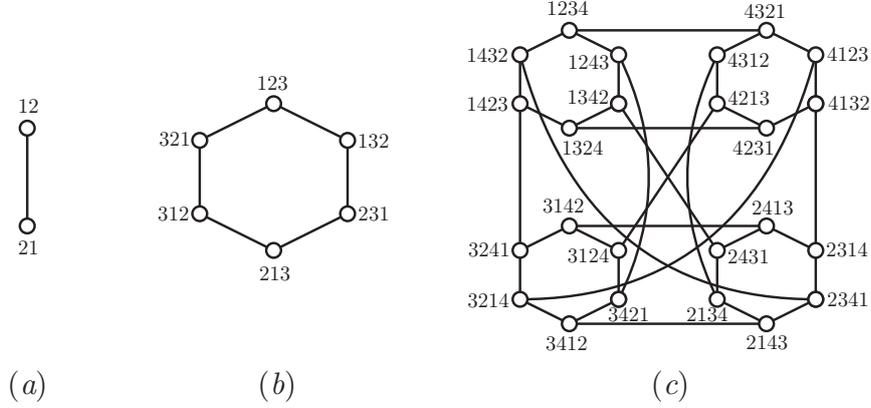}
	\end{center}
	\vskip-13cm
	\caption{Pancakes graphs: $(a)$ $P(2)$, $(b)$ $P(3)$, and $(c)$ $P(4)$.}
	\label{fig:pancakes}
\end{figure}

The {\it $n$-dimensional pancake graph}, proposed by Dweighter~\cite{d75} (see also Akers and Krishnamuthy~\cite{ak89}), and denoted by $P(n)$,
is a graph with the vertex set $V(P(n))=\big\{x_1x_2\ldots x_n|x_i\in [n],$ $
x_i\neq x_j\mbox{ for } i\neq j\big\}$. Its adjacencies are as follows:

\begin{equation}
x_1x_2\ldots x_n\ \sim\
\left\{
\begin{array}{l}
x_1\ldots x_{n-2} x_{n}x_{n-1},\\
x_1\ldots x_{n-3}x_nx_{n-1}x_{n-2},\\
x_1\ldots x_{n-4}x_nx_{n-1}x_{n-2}x_{n-3},\\
\hskip1cm \vdots \\
x_{n}x_{n-1}\ldots x_{2}x_{1}.
\end{array}
\right.
\label{eq:adj-pancake}
\end{equation}

The pancake graph $P(n)$ is a vertex-transitive $(n-1)$-regular graph with $n!$ vertices. It is a Cayley graph $\Cay(G,S)$, where $G$ is the symmetric group $Sym(n)$ and the generating set $S$ corresponds to the permutations of $x_1,x_2,\ldots, x_n$ given by \eqref{eq:adj-pancake}.
As examples, the pancakes graphs $P(2)$, $P(3)$, and $P(4)$ are shown in Figure~\ref{fig:pancakes}.

The exact diameters $k=k(n)$ of $P(n)$ are only known for $n\le 17$, as shown in Table~\ref{k(n)-P(n)} (see Cibulka \cite{cib11}
and Sloane~\cite{Sl07}). The best results to our knowledge were given by Gates and Papadimitriou \cite{gp79}, who proved that
\begin{equation*}
\label{bounds-k(n)}
\frac{17}{16}n\le k(n) \le \frac{5n+5}{3},
\end{equation*}
and by Heydari and Sudborough \cite{h97}, who improved the lower bound to
\begin{equation*}
k(n)\ge \frac{15}{14}n.
\end{equation*}

\begin{table}[t]
	\begin{center}
		\begin{tabular}{|c|ccccccccccccccccc|}
			\hline
			$n$ & 1 & 2 & 3 & 4 & 5 & 6 & 7 & 8 & 9  & 10 & 11 & 12 & 13 & 14 & 15 & 16 & 17 \\
			\hline
			$k$ & 0 & 1 & 3 & 4 & 5 & 7 & 8 & 9 & 10 & 11 & 13 & 14 & 15 & 16 & 17 & 18 & 19\\
			\hline
		\end{tabular}
	\end{center}
	\vskip-.25cm
	\caption{The known values of the diameter $k$ of the pancake graph $P(n)$.}
	\label{k(n)-P(n)}
\end{table}

\subsection{The new mixed graphs $\Gamma(d,n,r)$}

Recently, the pancake graphs, together with the cyclic prefix digraphs, were used by the first author \cite{Da19} to propose a new general family of mixed graphs.

The definition of these new mixed graphs is as follows.
\begin{definition}
	Given the integers $n\ge 2$ and $d,r\ge 1$, with $r<n\leq d+1$, the mixed graph $\G(d,n,r)$ has as vertex set the $n$-permutations of the $d+1$ symbols $1,2,\ldots,d,d+1$.
	Moreover, a vertex $x_1x_2\ldots x_n$ is adjacent, through edges, to the $r$ vertices
	\begin{equation}
	x_1x_2\ldots x_n\ \sim\
	\left\{
	\begin{array}{l}
	x_1x_2\ldots x_{n-2}x_{n}x_{n-1}\\
	x_1x_2\ldots x_{n-3}x_n x_{n-1}x_{n-2}\\
	\hskip 1cm\vdots \\
	x_1\ldots x_{n-r-1}x_n x_{n-1}\ldots x_{n-r}
	\end{array}
	\right.
	\label{edges}
	\end{equation}
	and adjacent, through arcs, to the $z=d-r$ vertices
	\begin{equation}
	x_1x_2\ldots x_n\ \rightarrow\
	\left\{
	\begin{array}{l}
	x_2x_3\ldots x_{n}y,\quad y\neq x_i,\ i=1,\ldots,n\quad \mbox{$($$d-n+1$ vertices$)$}\\
	\left.
	\begin{array}{l}
	x_1\ldots x_{n-r-2}x_{n-r}\ldots x_{n}x_{n-r-1}\\
	x_1\ldots x_{n-r-3}x_{n-r-1}\ldots x_nx_{n-r-2}\\
	\hskip1cm\vdots \\
	x_2\ldots x_{n}x_{1}
	\end{array}
	\right\}\ \mbox{$($$n-r-1$ vertices$)$.}
	\end{array}
	\right.
	\label{arcs}
	\end{equation}
\end{definition}

\begin{figure}[t]
	\begin{center}
		\includegraphics[width=14cm]{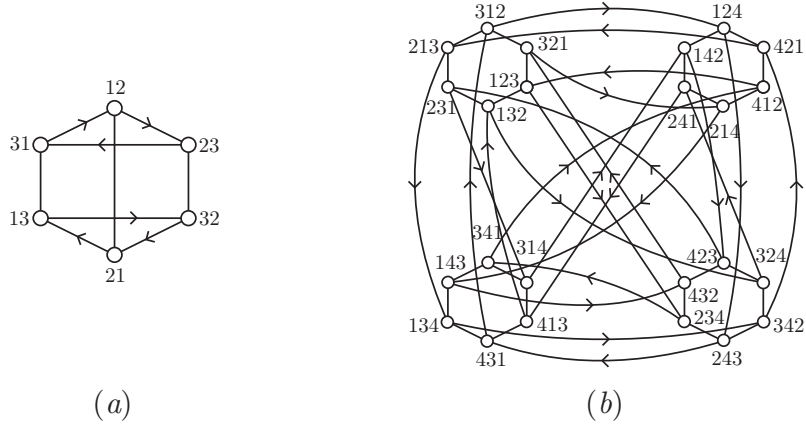}
	\end{center}
	\vskip-12.75cm
	\caption{New mixed graphs: $(a)$ $\Gamma(2,2,1)$ and $(b)$ $\Gamma(3,3,2)$.}
	\label{fig:nous-mixed}
\end{figure}

Thus, the number of vertices of the mixed graph $\G(d,n,r)$ is the number of $n$-permuta\-tions of $d+1$ elements,
$N=\frac{(d+1)!}{(d+1-n)!}$.	
Moreover, $\G(d,n,r)$ is a totally $(r,z)$-regular mixed graph, and it is also vertex-transitive.
In particular, if $n=d+1$ and $r=d$, then $\G(n-1,n,n-1)$ is the pancake graph $P(n)$; and if $r=1$, then $\G(d,n,1)$ coincides with the so-called cycle prefix digraph $\G_d(n)$ (notice that in this case, we require that $d\geq n$), see Faber, Moore, and Chen \cite{fmc93} or Comellas and Fiol~\cite{cf95}.
As examples, the new mixed graphs $\Gamma(2,2,1)$ and $\Gamma(3,3,2)$ are depicted in Figure \ref{fig:nous-mixed}.

\section{Regular partitions of vertices from number partitions}

Given a permutation $\pi:[n]\rightarrow [n]$, we denote by $P(\pi)=(p_{ij})$ the $n\times n$ permutation matrix with entries
$p_{ij}=1$ if $\pi(i)=j$, and $0$ otherwise (that is, the so-called {\em column representation}).

\begin{proposition}
\label{propo1}
Let $\Gamma=\Cay(G,S)$ be a Cayley digraph, where $G$ is a subgroup of the symmetric group $\Sym(n)$ and its generating set $S$ is given by the permutations $\{\pi_1,\pi_2,\ldots,$ $\pi_k\}$. Then,  $\Gamma$ has a regular partition  $\beta$ with quotient matrix $B=\sum_{i=1}^{k} P(\pi_i)$.
\end{proposition}

\begin{proof}
Let us show that the cells of the regular partition $\beta$ are the sets $V_i$, for $i=1,\ldots,n$, constituted by the permutations with a given digit, say $1$, in the fixed position $i$, that is $V_i=\{\pi\in S:\pi(i)=1\}$.
Indeed, if $u\in V_i$, the number of vertices $|\G^+(u)\cap V_j|$ (adjacent from $u$ and belonging to $V_j$) corresponds to the number of the permutations in $S$ that sends $1$ from the position $i$ to position $j$. This is precisely the $(i,j)$-entry of the matrix $B$, which is independent of $u$.
\end{proof}

\begin{figure}[t]
	\begin{center}
		\includegraphics[width=14cm]{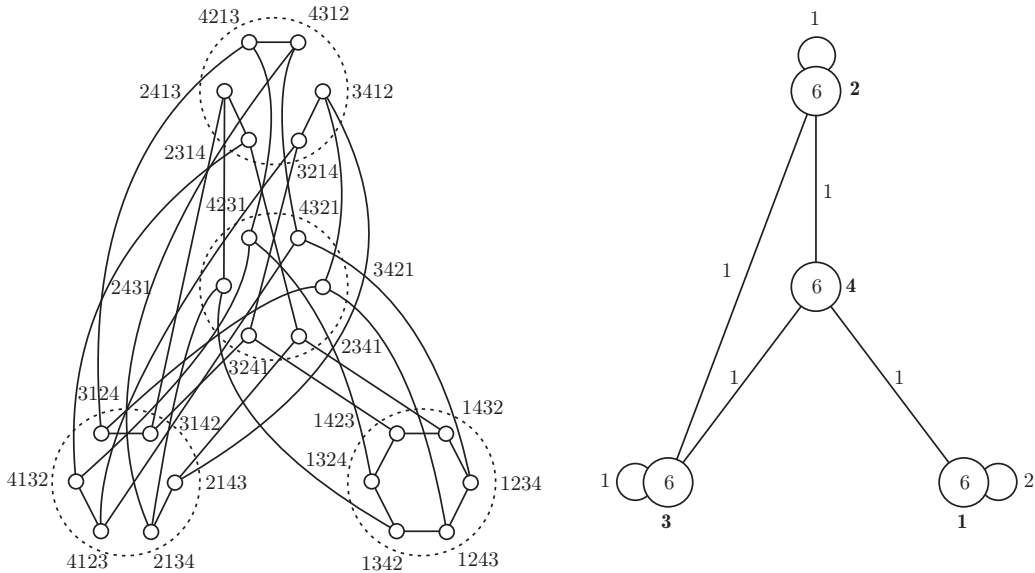}
	\end{center}
	\vskip-11cm
	\caption{A regular partition of the pancake graph $P(4)$ and its quotient graph. In boldface there is the numbering of the vertices (or classes).}
	\label{fig:particio-P(4)}
\end{figure}

As a corollary, we have a handy way of obtaining some of the eigenvalues of $\G$ since, by Lemma \ref{g93}$(ii)$,
$\spec B \subset \spec A(\Gamma).$


\begin{example*}[\textbf{Pancake graph} $P(4)$]
\label{ex:1}
Consider the pancake graph $P(4)$
as the Cayley graph $\Cay(S_4,S)$ with $S=\{(34),(24),(14)(23)\}$. Then, the sum of the corresponding permutation matrices $B=P((34))+P((24))+P((14)(23))$ turns out to be
$$
B=\left(
\begin{array}{cccc}
2 & 0 & 0 & 1\\
0 & 1 & 1 & 1\\
0 & 1 & 1 & 1\\
1 & 1 & 1 & 0
\end{array}
\right).
$$
According to Proposition \ref{propo1}, this is the quotient matrix of a regular partition of $P(4)$, as shown in Figure~\ref{fig:particio-P(4)}, together with its quotient graph. Notice that, as claimed, each class of vertices contains all the permutations with $1$ in a fixed position. Moreover, $\spec B=\{3,2,0,-1\}$, a part of the spectrum of $P(4)$ that, as we show in Section \ref{sec:digrafs-quocient}, it is
\begin{equation}
\label{sp(P4)}
\spec P(4)=\left\lbrace[3]^1,[2]^5,\left[\textstyle\frac{-1+\sqrt{17}}{2}\right]^3,[0]^5,\textstyle\left[ \frac{-1-\sqrt{17}}{2}\right]^3,[-1]^4,[-2]^3\right\rbrace.
\end{equation}
\end{example*}


\begin{example*}[\textbf{New mixed graph} $\Gamma(3,3,2)$]
		Consider the new mixed graph $\Gamma(3,3,2)$
	as the Cayley graph $\Cay(S_4,S)$ with $S=\{(34),(24),(2341)\}$. Now, the sum of the corresponding permutation matrices $B=P((34))+P((24))+P((2341))$ is
	$$
	B=\left(
	\begin{array}{cccc}
	2 & 0 & 0 & 1\\
	1 & 1 & 0 & 1\\
	0 & 1 & 1 & 1\\
	0 & 1 & 2 & 0
	\end{array}
	\right),
	$$
    that corresponds to the quotient matrix of the regular partition shown in Figure~\ref{fig:nous-mixed-part-reg}.
	Moreover, $\spec B=\{[3]^1,[1]^2,[-1]^2\}$. As expected, $\spec B\subset \spec\Gamma(3,3,2)$, since, as shown in Section \ref{sec:digrafs-quocient},
	\begin{equation}
	\label{eq:espectre-new-mixed}
	\spec\Gamma(3,3,2) = \{[3]^1,[\sqrt{3}]^2,[1]^9,[-1]^9,[-\sqrt{3}]^2,[-3]^1\}.
	\end{equation}
Note that this spectrum is symmetric, in concordance with the fact that $\Gamma(3,3,2)$ is a {\em bipartite} (mixed) graph.
\end{example*}

\begin{figure}[t]
	\begin{center}
		\includegraphics[width=14cm]{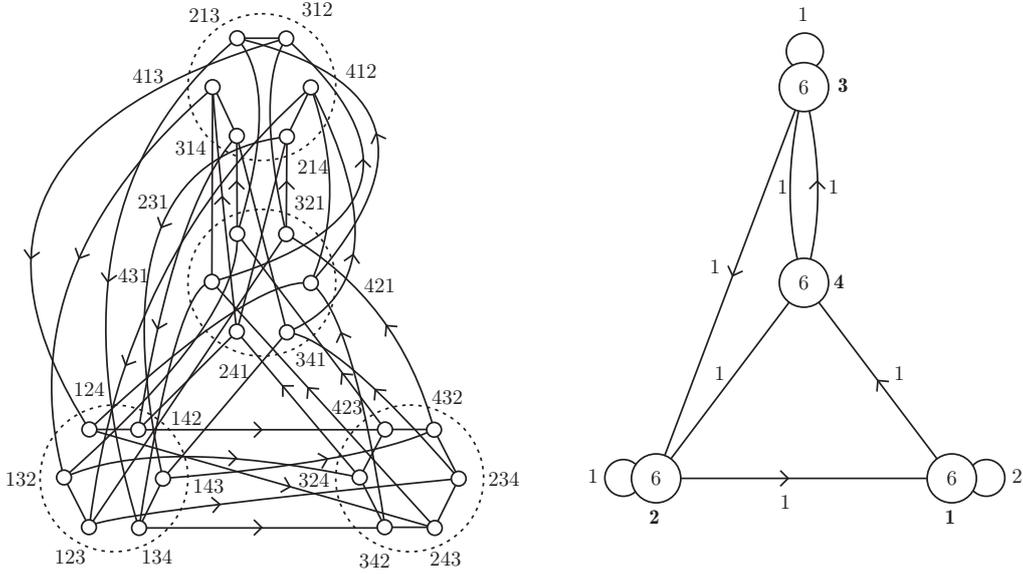}
	\end{center}
	\vskip-11cm
	\caption{The new mixed graph $\Gamma(3,3,2)$ drawn as a regular partition and its quotient graph. In boldface there is  the numbering of the vertices.}
	\label{fig:nous-mixed-part-reg}
\end{figure}


The result of the previous examples can be generalized to obtain some eigenvalues and their associated eigenvectors of the whole family of the pancake graphs $P(n)$ and the new mixed graphs $\G(n,n,n-1)$, as shown in the following results.

\begin{proposition}
The  matrix $B_n=\sum_{i=1}^{n} P(\pi_i)$ of the pancake graph $P(n)$ is the sum $B_n=D_n+T_n$, where $D_n=\diag(n-2,n-1,\ldots,0,-1)$
and $T_n$ is the `lower anti-triangular matrix' with entries $(T_n)_{ij}=1$ if $i+j\ge n+1$, and $(T_n)_{ij}=0$ otherwise,
with spectrum
$$
\spec B_n
=\{n-1,n-2,\ldots,0,-1\}\setminus \{\lfloor (n/2)-1\rfloor\}
$$
(all the eigenvalues with multiplicity one).
Moreover, their associated eigenvectors are, respectively,
the all-1 vector $(1,1,\ldots,1)^{\top}$,
\begin{equation}
\label{eigenvects1}
	 \Big( 0,\stackrel{(r-1)}{\ldots},0,n-2r,-1,\stackrel{(n-2r)}{\ldots},-1,0,\stackrel{(r)}{\ldots},0\Big)^{\top}, \mbox{ for }
\left\{
\begin{array}{ll}
	 r=1,\ldots,\lfloor n/2\rfloor & (n \mbox{ odd}), \\
	 r=1,\ldots,\lfloor n/2\rfloor-1 & (n \mbox{ even}),
\end{array}
 \right.
 \end{equation}
  and
\begin{equation}
\label{eigenvects2}
	\Big( 0,\stackrel{(r-1)}{\ldots},0,-1,\stackrel{(n-2r+1)}{\ldots},-1,
		n-2r+1,0,\stackrel{(r-1)}{\ldots},0\Big)^{\top}, \mbox{ for } r=\lfloor n/2\rfloor,\ldots,\ldots,1.
\end{equation}
\end{proposition}

\begin{proof}
First, it is straightforward to check that the matrix $B_n$ is as claimed. We just have to compute the sum of the involved permutation matrices. For example, for $n=5$, we get
\begin{align*}
&B_5={\scriptsize{
\left(
\begin{array}{ccccc}
1 & 0 & 0 & 0 & 0\\
0 & 1 & 0 & 0 & 0\\
0 & 0 & 1 & 0 & 0\\
0 & 0 & 0 & 0 & 1\\
0 & 0 & 0 & 1 & 0
\end{array}
\right)+
\left(
\begin{array}{ccccc}
1 & 0 & 0 & 0 & 0\\
0 & 1 & 0 & 0 & 0\\
0 & 0 & 0 & 0 & 1\\
0 & 0 & 0 & 1 & 0\\
0 & 0 & 1 & 0 & 0
\end{array}
\right)+
\left(
\begin{array}{ccccc}
1 & 0 & 0 & 0 & 0\\
0 & 0 & 0 & 0 & 1\\
0 & 0 & 0 & 1 & 0\\
0 & 0 & 1 & 0 & 0\\
0 & 1 & 0 & 0 & 0
\end{array}
\right)
+
\left(
\begin{array}{ccccc}
0 & 0 & 0 & 0 & 1\\
0 & 0 & 0 & 1 & 0\\
0 & 0 & 1 & 0 & 0\\
0 & 1 & 0 & 0 & 0\\
1 & 0 & 0 & 0 & 0
\end{array}
\right)}}
\\
 &
 ={\scriptsize{\left(
\begin{array}{ccccc}
3 & 0 & 0 & 0 & 1\\
0 & 2 & 0 & 1 & 1\\
0 & 0 & 2 & 1 & 1\\
0 & 1 & 1 & 1 & 1\\
1 & 1 & 1 & 1 & 0
\end{array}
\right)
 =\left(
\begin{array}{ccccc}
3 & 0 & 0 & 0 & 0\\
0 & 2 & 0 & 0 & 0\\
0 & 0 & 1 & 0 & 0\\
0 & 0 & 0 & 0 & 0\\
0 & 0 & 0 & 0 & -1
\end{array}
\right)
+\left(
\begin{array}{ccccc}
0 & 0 & 0 & 0 & 1\\
0 & 0 & 0 & 1 & 1\\
0 & 0 & 1 & 1 & 1\\
0 & 1 & 1 & 1 & 1\\
1 & 1 & 1 & 1 & 1
\end{array}
\right)}}
=D_5+T_5.
\end{align*}

Concerning
the eigenpairs, in the case that $n$ is even, let us check that, for every $r=1,\ldots,n/2$,
$$
v_r=\Big( 0,\stackrel{(r-1)}{\ldots},0,n-2r,-1,\stackrel{(n-2r)}{\ldots},-1,0,\stackrel{(r)}{\ldots},0\Big)^{\top}
$$
is an eigenvector with eigenvalue $\lambda_r=n-r-1$:

\begin{align*}
&B_n v_k =D_n v_k+T_n v_k\\
&=\Big( 0,\stackrel{(r-1)}{\ldots},0,(n-2r)(n-r-1),-(n-r-2),-(n-r-3),\ldots,-(r-1),0,\stackrel{(r)}{\ldots},0\Big)^{\top}\\
&+\Big( 0,\stackrel{(r)}{\ldots},0,-1,-2,\ldots,-(n-2r),0,\stackrel{(r)}{\ldots},0\Big)^{\top}\\
&=\Big( 0,\stackrel{(r-1)}{\ldots},0,(n-2r)\lambda_r,-\lambda_r+1,-\lambda_r+2,-\lambda_r+3,\ldots,-\lambda_r+(n-2r),0,\stackrel{(r)}{\ldots},0\Big)^{\top}\\
&+\Big( 0,\stackrel{(r)}{\ldots},0,-1,-2,\ldots,-(n-2r),0,\stackrel{(r)}{\ldots},0\Big)^{\top}\\
&=\lambda_r\Big( 0,\stackrel{(r-1)}{\ldots},0,n-2r,-1,\stackrel{(n-2r)}{\ldots},-1,0,\stackrel{(r)}{\ldots},0\Big)^{\top}=\lambda_r v_r.
\end{align*}
The other eigenpairs and the case for odd $n$ can be proved analogously.
\end{proof}
For example, for the case $n=5$, we obtain
$\spec B_5=\{4,3,2,0,-1\}$,
with corresponding matrix of (column) eigenvectors
$$
\left(
\begin{array}{rrrrr}
	1 & 3 & 0 & 0 & -1\\
	1 & -1 & 1 & -1 & -1\\
	1 & -1 & -1 & -1 & -1\\
	1 & -1 & 0 & 2 & -1\\
	1  & 0 & 0 & 0 & 4
\end{array}
\right).
$$
Looking at last column, Lemma \ref{g93}$(ii)$ implies that $P(n)$ has the eigenvalue $-1$ whose eigenvector has entries $-1$ for the vertices of the classes $1,2,\ldots,n-1$, and entries $n-1$ for the vertices of the class $n$ (permutations with last symbol $1$, as shown in Figure \ref{fig:particio-P(4)} for the case $n=4$).
In this situation, it is known that the graph has a {\em perfect code} or {\em efficient dominating set} $C$ (that is, $C$ is an independent vertex set, and each vertex not in $C$ is adjacent to exactly one vertex in $C$). In other words, to each perfect code $C$ corresponds a $(-1)$-eigenvector as described. See, for instance, Godsil \cite{g93}.
Then, by using the result by Konstantinova \cite[Thm. 1]{k13}, who proved that the pancake graph $P(n)$ contains exactly $n$ perfect codes (the sets of vertices with the same last symbol), we get the following result. First, recall that a circulant matrix $\Circ (a_1,a_2,\ldots,a_n)$ has first row
$a_1,a_2,\ldots,a_n$ and, for $i=2,\ldots,n$, its $i$-th row  is obtained from the $(i-1)$-th row by cyclically shifting it to the right one position.

\begin{lemma}
The pancake graph $P(n)$ has eigenvalue $-1$ with multiplicity $m(-1)\ge n-1$.
\end{lemma}
\begin{proof}
Each of the $n$ different perfect codes induces a regular partition with quotient matrix having a $(-1)$-eigenvector as above. Then, since $\rank \Circ(n-1,-1,\stackrel{(n-1)}{\ldots},-1)=n-1$, we conclude that $n-1$ of such eigenvectors are independent.
\end{proof}

In the case of the new mixed graph $\Gamma(n,n,n-1)$, we have the following result. 
\begin{proposition}
\label{propo2}
The  matrix $B'_n=\sum_{i=1}^{n} P(\pi_i)$ of the new mixed graph $\Gamma(n,n,n-1)$, $n\ge 3$,  is
the sum $B'_n=D_n+C_n+T'_n$, where $D_n=\diag(n-2,n-1,\ldots,0,-1)$,
$C_n$ is the circulant matrix $\Circ(0,1,0,\ldots,0)$, and $T'_n$ is the `lower anti-triangular matrix' with entries $(T'_n)_{ij}=1$ if $i+j > n+1$ and $(T'_n)_{ij}=0$ otherwise,
with eigenvalues
$
\{n-1,n-3,-1\}\subset \spec B'_n.
$
\end{proposition}

\begin{proof}
Again, by computing the sum of the involved permutation matrices, it is easy to check that the matrix $B'_n$ is as claimed. For example, for $n=5$, we get
\begin{align*}
&B'_5={\scriptsize{
\left(
\begin{array}{ccccc}
1 & 0 & 0 & 0 & 0\\
0 & 1 & 0 & 0 & 0\\
0 & 0 & 1 & 0 & 0\\
0 & 0 & 0 & 0 & 1\\
0 & 0 & 0 & 1 & 0
\end{array}
\right)+
\left(
\begin{array}{ccccc}
1 & 0 & 0 & 0 & 0\\
0 & 1 & 0 & 0 & 0\\
0 & 0 & 0 & 0 & 1\\
0 & 0 & 0 & 1 & 0\\
0 & 0 & 1 & 0 & 0
\end{array}
\right)+
\left(
\begin{array}{ccccc}
1 & 0 & 0 & 0 & 0\\
0 & 0 & 0 & 0 & 1\\
0 & 0 & 0 & 1 & 0\\
0 & 0 & 1 & 0 & 0\\
0 & 1 & 0 & 0 & 0
\end{array}
\right)
+
\left(
\begin{array}{ccccc}
0 & 0 & 0 & 0 & 1\\
0 & 1 & 0 & 0 & 0\\
0 & 0 & 1 & 0 & 0\\
0 & 0 & 0 & 1 & 0\\
0 & 0 & 0 & 0 & 1
\end{array}
\right)}}
\\
 &
 ={\scriptsize{\left(
\begin{array}{ccccc}
3 & 0 & 0 & 0 & 1\\
1 & 2 & 0 & 0 & 1\\
0 & 1 & 1 & 1 & 1\\
0 & 0 & 2 & 1 & 1\\
0 & 1 & 1 & 2 & 0
\end{array}
\right)
 =\left(
\begin{array}{ccccc}
3 & 0 & 0 & 0 & 0\\
0 & 2 & 0 & 0 & 0\\
0 & 0 & 1 & 0 & 0\\
0 & 0 & 0 & 0 & 0\\
0 & 0 & 0 & 0 & -1
\end{array}
\right)
+\left(
\begin{array}{ccccc}
0 & 0 & 0 & 0 & 1\\
0 & 1 & 0 & 0 & 0\\
0 & 0 & 1 & 0 & 0\\
0 & 0 & 0 & 1 & 0\\
0 & 0 & 0 & 0 & 1
\end{array}
\right)
+\left(
\begin{array}{ccccc}
0 & 0 & 0 & 0 & 0\\
0 & 0 & 0 & 0 & 1\\
0 & 0 & 0 & 1 & 1\\
0 & 0 & 1 & 1 & 1\\
0 & 1 & 1 & 1 & 1
\end{array}
\right)}}\\
&
=D_5+C_5+T'_5.
\end{align*}
Concerning
the eigenvalues, it is readily checked  the vectors $(1,1,\ldots,1)^{\top}$  and $(-1,\ldots,-1,$ $n-1)^{\top}$ are eigenvectors  of $B'_n$ with eigenvalues $n-1$ and $-1$ respectively.
Alternatively, to prove that $-1\in \spec B'_n$, we can also check that $\det (-I-B'_n)=(-1)^n det (I+B'_n)=0$. Note that this holds since $n$ times the last column of $I+B'_n$ is the sum of the other columns. For instance, for $n=5$, we get
$$
B_5'+I=\left(
\begin{array}{ccccc}
4 & 0 & 0 & 0 & 1\\
1 & 3 & 0 & 0 & 1\\
0 & 1 & 2 & 1 & 1\\
0 & 0 & 2 & 2 & 1\\
0 & 1 & 1 & 2 & 1
\end{array}
\right).
$$
Finally, $n-3$ is an eigenvalue of $B'_n$ since.the first two rows of $(n-3)I-B'_n$ are equal. Namely,
$(-1,0,\ldots,0,-1)$.
\end{proof}

In fact, the first statments of Propositions \ref{propo1} and \ref{propo2} are particular cases of the following result. Let ${\cal PR}=PR_n^{n_1,\ldots,n_r}$ denote the set of permutations with repetitions of $r$ symbols $a,b,\ldots$, where $a$ is repeated $n_1$ times, $b$ is repeated $n_2$ times, etc. Thus, $|{\cal PR}|=\frac{n!}{n_1!\cdots n_r!}$.

\begin{theorem}
\label{th:sigma-tau}
Let $\Gamma=\Cay(G,S)$ be a Cayley digraph, where $G$ is a subgroup of the symmetric group $Sym(n)$ and its generating set $S$ is given by the permutations $\{\pi_1,\pi_2,\ldots,\pi_k\}$. For any partition $n_1+n_2+\cdots+n_r=n$, there is a regular partition of $\Gamma$ with quotient matrix $B$ indexed by  the elements of ${\cal PR}$, and for every $\sigma,\tau\in {\cal PR}$ the entry $(B)_{\sigma\tau}$ is the number (possibly zero) of permutations in $S$ that, acting on the symbol positions, map $\sigma$ into $\tau$.
\end{theorem}

\begin{proof}
Given the partition $n_1+n_2+\cdots+n_r=n$, we define the onto mapping $\phi:[n]\rightarrow \{a_1,\ldots a_r\}$ such that $|\phi^{-1}(a_i)|=n_i$, for $i=1,\ldots,r$. Given a permutation $\pi\in Sym(n)$, let $\phi\cdot \pi$ be the permutation with repetition in
${\cal PR}$ with $i$-th symbol $(\phi\cdot \pi)(i)=\phi(\pi(i))$, for $i=1,\ldots,n$. If we let $\pi$ to act on the symbol positions (composition of permutations $gh$ being read from left to right), then we can also define the permutation with repetition $\pi\cdot \phi$ such that $(\pi\cdot \phi)(i)=\pi(\phi(i))=\phi(\pi(i))$ for $i=1,\ldots,n$ and, hence, $\phi\cdot \pi=\pi\cdot \phi$. Also, it is clear that, for any $g,h\in G$, 
$(g h)\cdot\phi=g\cdot (h\cdot \phi)$. Let $\phi(G)$ be the set of distinct permutations in ${\cal PR}$ of the form $\phi\cdot g$ for some $g\in G$. Now we claim that $\Gamma$ has a regular partition $\phi^*$, where each class $V_{\sigma}$ is represented by an element $\sigma\in  \phi(G)$. More precisely, $V_{\sigma}=\{g\in G : \phi\cdot g=\sigma\}$. Indeed, if $\phi\cdot g= \phi\cdot g'$ and $g \rightarrow \pi g$ for some $\pi\in S$, we have
\begin{align}
\phi\cdot (\pi g)& =(\pi g)\cdot \phi=\pi\cdot (g\cdot \phi)=
\pi\cdot (\phi \cdot g ) \label{phi-pi-g}\\
 &=
\pi\cdot (\phi \cdot g' )=\pi\cdot (g'\cdot \phi)=(\pi g')\cdot \phi
=\phi\cdot (\pi g'). \nonumber
\end{align}
Thus, $\phi$ can be interpreted as a homomorphism from $\G$ to its quotient digraph $\phi^*(\G)$ that preserves the `colors' (generators) of the arcs.
The corresponding quotient matrix $B$ is then indexed by  elements of $\phi(G)\subset {\cal PR}$, with entries  $(B)_{\sigma\tau}$ for every $\sigma,\tau\in {\cal PR}$, as claimed.
\end{proof}

\begin{example*}[\textbf{Pancake graph} $P(4)$]
Consider the pancake graph $P(4)$. In this case, we have the following partitions: $4$, $3+1$, $2+2$, $2+1+1$, and $1+1+1+1$. According to Theorem \ref{th:sigma-tau}, these partitions yield the regular partitions of $P(4)$ in Figure \ref{fig:particionsP(4)}, with number of classes $PR_4^4=1$, $PR_4^{3,1}=4$, $PR_4^{2,2}=6$, and $PR_4^{2,1,1}=12$, respectively. Note that the classes are identified with the corresponding permutations with repetition of the symbols $a,b,c$. Besides, observe that the case of the previous example of $P(4)$ corresponds to the partition $3+1$.
Note that the graph associated with the partition $1+1+1+1$ is the whole graph $P(4)$, with number of classes (that is, number of vertices) $PR_4^{1,1,1,1}=24$ (see Figure \ref{fig:pancakes}$(c)$).
\end{example*}

\begin{figure}[t]
	\begin{center}
		\includegraphics[width=14cm]{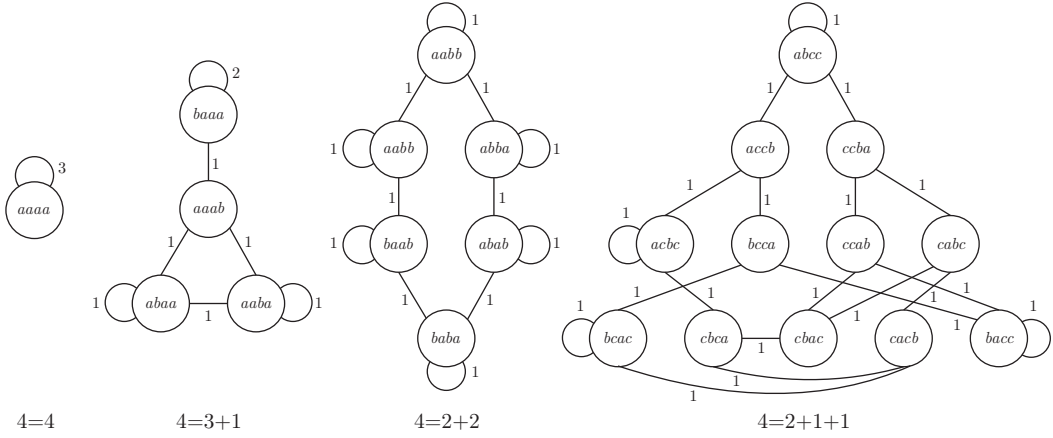}
	\end{center}
	\vskip-5.5cm
	\caption{The regular partitions of $P(4)$ corresponding to the number partitions of 4.}
	\label{fig:particionsP(4)}
\end{figure}

\begin{example*}[\textbf{New mixed graph} $\Gamma(3,3,2)$]
	Consider the new mixed graph $\Gamma(3,3,2)$. We have the partitions $4$, $3+1$, $2+2$, $2+1+1$, and $1+1+1+1$ again. These partitions give the regular partitions of $\Gamma(3,3,2)$ in Figure \ref{fig:particionsGamma(3,3,2)}, with the same number of classes that in the previous example of $P(4)$. The classes are again identified with the corresponding permutations with repetition of the symbols $a,b,c$. Note that the case of the previous example of $\Gamma(3,3,2)$ corresponds to the partition $3+1$.
	The graph associated with the partition $1+1+1+1$ is the whole graph $\Gamma(3,3,2)$ on 24 vertices (see Figure \ref{fig:nous-mixed}$(b)$).
\end{example*}

\begin{figure}[t]
	\begin{center}
		\includegraphics[width=14cm]{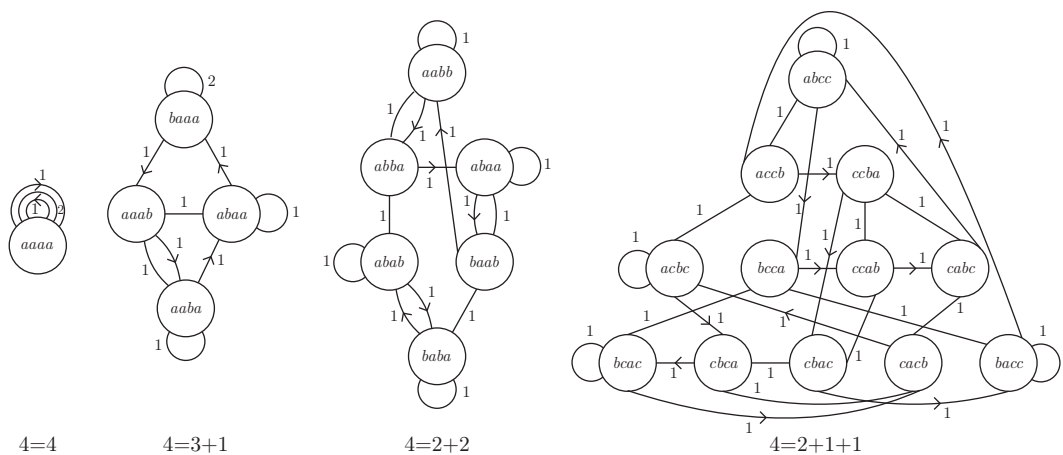}
	\end{center}
	\vskip-5cm
	\caption{The regular partitions of $\Gamma(3,3,2)$ corresponding to the number partitions of 4.}
	\label{fig:particionsGamma(3,3,2)}
\end{figure}


\section{The spectra  of quotient digraphs}
\label{sec:digrafs-quocient}
In the previous section, we found some eigenvalues, eigenvectors, and regular partitions for whole families of digraphs. In this section, we give 
a method 
to find the whole spectrum of a Cayley digraph (on a permutation group) and their quotient digraphs associated with the corresponding partitions. We again illustrate the obtained results by using the previous examples: the pancake graph $P(4)$ and the new mixed graph $\Gamma(3,3,2)$.

Our results are based on the following lemma, which allows us to apply Theorem \ref{theo-sp2}.
\begin{lemma}
\label{th:sigma-tau-sp}
Let $\Gamma=\Cay(G,S)$ be a Cayley digraph, where $G$ is a subgroup of the symmetric group $Sym(n)$ with generating set $S=\{\pi_1,\pi_2,\ldots,\pi_k\}$. For a given partition $n_1+n_2+\cdots+n_r=n$ induced by the mapping $\phi$, the quotient digraph $\phi^*(\G)$ is isomorphic to the relative lift
$\Gamma^{\beta}$ with base digraph a singleton with arcs $a_1,\ldots,a_k$, group $G$, relative voltage assignment  $\beta$  defined by $\beta(a_i)=\pi_i$ for $i=1,\ldots,k$, and stabilizer subgroup 
$$
H=\Stab_G(V_1)\cap \cdots \cap \Stab_G(V_r),
$$
where $V_1\cup\cdots \cup V_r$ is a partition of $[n]$ with $|V_i|=n_i$ for $i=1,\ldots,r$.
\end{lemma}
\begin{proof}
Let $\phi$ be the mapping defined in the proof of Theorem \ref{th:sigma-tau}.
Let $e$ be the identity element of $G$, and assume that $\phi\cdot e$ is the permutation with repetition $\varepsilon=a_1\stackrel{(n_1)}{\dots} a_1\ldots a_r\stackrel{(n_r)}{\dots} a_r$. Then, $H$ is constituted by all the elements $h\in G$ such that $\phi\cdot h=\varepsilon$ and, in general, each left coset of $H$ is of the form
	$$
	gH=\{h\in G : \phi\cdot h =\sigma\}\quad\mbox{if}\quad \phi\cdot g=\sigma.
	$$
Thus,  $gH$ corresponds to the class, or vertex of $\phi^*(\G)$,  $V_{\sigma}=\{g\in G : \phi\cdot g=\sigma\}$ with $\sigma\in \phi(G)$. Moreover, $V_{\sigma}$ is adjacent to $V_{\tau}$ through an arc with `color' $\pi\in S$ if $\tau=\phi\cdot(\pi g)= \pi\cdot \sigma$ (where the second equality comes from \eqref{phi-pi-g}). Consequently, in the relative lift $\G^{\beta}$, the vertex $gH$ is adjacent, through the arc with `color' $\pi$, to $\pi g H$. This proves the claimed isomorphism.
\end{proof}



\begin{example*}[\textbf{Pancake graph} $P(4)$]
Consider the case of the pancake graph $P(4)$ again.
We begin computing the spectrum of the whole graph by using Theorem \ref{theo-sp}.
We obtained from SageMath the matrices of the irreducible representations of $S_4$, that is, $\rho_1$ (partition $4=4$), $\rho_2$ (partition $4=3+1$),  $\rho_3$ (partition $4=2+2$),  $\rho_4$ (partition $4=2+1+1$), and  $\rho_5$ (partition $4=1+1+1+1$) shown in Table \ref{taula:S_4_matrius_irr}, related to the permutations $a=1243$, $b=1432$, and $c=4321$. Then, from Theorem \ref{theo-sp}, the spectrum of $P(4)$ is the union of the following spectra
(Note that the dimension of the matrices gives the multiplicities of the corresponding eigenvalues).
\begin{enumerate}
\item[$(i)$]
$1\cdot \spec \rho_1(B)=\{[3]^1\}$,
\item[$(ii)$]
$3\cdot \spec \rho_2(B)=\{[2]^3, [0]^3, [-1]^3\}$,
\item[$(iii)$]
$2\cdot \spec \rho_3(B)=\{[2]^2, [0]^2\}$,
\item[$(iv)$]
$3\cdot \spec \rho_4(B)=\{[\frac{-1+\sqrt{17}}{2}]^3, [-2]^3, \frac{-1-\sqrt{17}}{2}]^3\}$,
\item[$(v)$]
$1\cdot \spec \rho_5(B)=\{[-1]^1\}$,
\end{enumerate}
giving $\spec(P_4)$ as claimed in \eqref{sp(P4)}.

\begin{table}[h!]
	\small
	\medskip\noindent
	\begin{center}
		\begin{tabular}{|cccc|}
			\hline
			\multicolumn{1}{|c|}{Partition:} & $a$: & $b$: & $c$: \\
			\multicolumn{1}{|c|}{4=4} & (1) & (1) & (1)     \\[0.2cm]
			\cline{1-1}
			$a+b+c$: & Dimension: & Eigenvalues: & Spectrum: \\
			(3) & 1 & 3 & $[3]^1$ \\
			\hline\hline
			\multicolumn{1}{|c|}{Partition:} & $a$: & $b$: & $c$: \\
			\multicolumn{1}{|c|}{4=3+1} &
			$\left( \begin{array}{ccc}
			1 & 0 & 0\\
			0 & 1 & 0\\
			-1 & 1 &-1
			\end{array}\right)$         &
			$\left( \begin{array}{ccc}
			1 & 0 & 0\\
			1 & -1 & 1\\
			0 & 0 & 1
			\end{array}\right)$        &
			$\left( \begin{array}{ccc}
			-1 & 1 & -1\\
			0 & 0 & -1\\
			0 & -1 & 0
			\end{array}\right)$ \\[0.7cm]
			\cline{1-1}
			$a+b+c$: & Dimension: & Eigenvalues: & Spectrum: \\
			$\left( \begin{array}{ccc}
			1 & 1 & -1\\
			1 & 0 & 0\\
			-1 & 0 & 0
			\end{array}\right)$  &  3 & $2,0,-1$ & $[2]^3,[0]^3,[-1]^3$\\
			\hline\hline
			\multicolumn{1}{|c|}{Partition:} & $a$: & $b$: & $c$: \\
			\multicolumn{1}{|c|}{4=2+2} &
			$\left( \begin{array}{cc}
			1 & 0 \\
			1 & -1
			\end{array}\right)$         &
			$\left( \begin{array}{cc}
			-1 & 1 \\
			0 & 1
			\end{array}\right)$        &
			$\left( \begin{array}{cc}
			1 & 0 \\
			0 & 1
			\end{array}\right)$ \\[0.7cm]
			\cline{1-1}
			$a+b+c$: & Dimension: & Eigenvalues: & Spectrum: \\	
			$\left( \begin{array}{cc}
			1 & 1 \\
			1 & 1
			\end{array}\right)$      &  2 & $2,0$ & $[2]^2,[0]^2$\\
			\hline\hline
			\multicolumn{1}{|c|}{Partition:} & $a$: & $b$: & $c$: \\
			\multicolumn{1}{|c|}{4=2+1+1} &
			$\left( \begin{array}{ccc}
			1 & 0 & 0\\
			1 & -1 & 0\\
			1 & 0 & -1
			\end{array}\right)$         &
			$\left( \begin{array}{ccc}
			-1 & 1 & 0\\
			0 & 1 & 0\\
			0 & 1 & -1
			\end{array}\right)$        &
			$\left( \begin{array}{ccc}
			0 & 1 & -1\\
			1 & 0 & -1\\
			0 & 0 & -1
			\end{array}\right)$ \\[0.7cm]
			\cline{1-1}
			$a+b+c$: & Dimension: & Eigenvalues: & Spectrum: \\	
			$\left( \begin{array}{ccc}
			0 & 2 & -1\\
			2 & 0 & -1\\
			1 & 1 & -3
			\end{array}\right)$ & 3 & $\frac{-1+\sqrt{17}}{2},-2,\frac{-1-\sqrt{17}}{2}$ & $[\frac{-1+\sqrt{17}}{2}]^3,[-2]^3,[\frac{-1-\sqrt{17}}{2}]^3$     \\
			\hline\hline
			\multicolumn{1}{|c|}{Partition:} & $a$: & $b$: & $c$: \\
			\multicolumn{1}{|c|}{4=1+1+1+1} & (-1)        & (-1)        & (1) \\[0.2cm]
			\cline{1-1}
			Sum: & Dimension: & Eigenvalues: & Spectrum: \\
			(-1) & 1 & -1 & $[-1]^1$  \\
			\hline
		\end{tabular}
		\caption{The irreducible matrices of $P(4)$, their sum, and their corresponding eigenvalues.}
		\label{taula:S_4_matrius_irr}
	\end{center}
\end{table}
Now to find the spectra of the different quotient graphs, induced from each partition, from Theorem \ref{theo-sp2} and Lemma \ref{th:sigma-tau-sp}, we need to know the ranks of $\rho_i(H_j)$ for all group stabilizers $H_j$. With this aim, we use the matrices of all irreducible representation shown in Table \ref{taula:S_4_permutacions}.
\begin{table}[h!]
	\small
	\noindent
	\begin{center}
		\begin{tabular}{|c|c|c|c|c|c|c|}
			\hline
			Permu- & Cyclic   & $4$       &  $3+1$    & $2+2$    & $2+1+1$   & $1+\cdots+1$\\
			tation & Notation & $\rho_1$  &  $\rho_2$ & $\rho_3$ &  $\rho_4$        &  $\rho_5$         \\
			\hline\hline
			1234        & $e$      & 1     &
			$\tiny{\left(\begin{array}{ccc}
				1 & 0 & 0\\
				0 & 1 & 0\\
				0 & 0 & 1
				\end{array}\right)}$ &
			$\tiny{\left( \begin{array}{cc}
				1 & 0\\
				0 & 1\\
				\end{array}\right)}$ &
			$\tiny{\left( \begin{array}{ccc}
				1 & 0 & 0\\
				0 & 1 & 0\\
				0 & 0 & 1
				\end{array}\right)}$ & 1\\
			\hline
			2134        & $(12)$      & 1     &
			$\tiny{\left( \begin{array}{ccc}
				0 & -1 & 0\\
				-1 & 0 & 0\\
				1 & 0 & 0
				\end{array}\right)}$ &
			$\tiny{\left( \begin{array}{cc}
				1 & 0\\
				1 & -1\\
				\end{array}\right)}$ &
			$\tiny{\left(\begin{array}{ccc}
				-1 & 0 & 0\\
				0 & 0 & -1\\
				0 & -1 & 0
				\end{array}\right)}$ & -1\\
			\hline
			3124        & $(132)$      & 1     &
			$\tiny{\left( \begin{array}{ccc}
				0 & -1 & 0\\
				0 & 0 & -1\\
				1 & 0 & 0
				\end{array}\right)}$ &
			$\tiny{\left( \begin{array}{cc}
				-1 & 1\\
				-1 & 0\\
				\end{array}\right)}$ &
			$\tiny{\left(\begin{array}{ccc}
				0 & 0 & 1\\
				1 & 0 & 0\\
				0 & 1 & 0
				\end{array}\right)}$ & 1\\
			\hline
			1324        & $(23)$      & 1     &
			$\tiny{\left( \begin{array}{ccc}
				1 & 0 & 0\\
				0 & 0 & -1\\
				0 & -1 & 0
				\end{array}\right)}$ &
			$\tiny{\left( \begin{array}{cc}
				0 & -1\\
				-1 & 0\\
				\end{array}\right)}$ &
			$\tiny{\left(\begin{array}{ccc}
				0 & -1 & 0\\
				-1 & 0 & -1\\
				0 & 0 & -1
				\end{array}\right)}$ & -1\\
			\hline
			2314        & $(123)$      & 1     &
			$\tiny{\left( \begin{array}{ccc}
				0 & 0 & 1\\
				-1 & 0 & 0\\
				0 & -1 & 0
				\end{array}\right)}$ &
			$\tiny{\left( \begin{array}{cc}
				0 & -1\\
				1 & -1\\
				\end{array}\right)}$ &
			$\tiny{\left(\begin{array}{ccc}
				0 & 1 & 0\\
				0 & 0 & 1\\
				1 & 0 & 0
				\end{array}\right)}$ & 1\\
			\hline
			3214        & $(13)$      & 1     &
			$\tiny{\left( \begin{array}{ccc}
				0 & 0 & 1\\
				0 & 1 & 0\\
				1 & 0 & 0
				\end{array}\right)}$ &
			$\tiny{\left( \begin{array}{cc}
				-1 & 1\\
				0 & 1\\
				\end{array}\right)}$ &
			$\tiny{\left(\begin{array}{ccc}
				0 & 0 & -1\\
				0 & -1 & 0\\
				-1 & 0 & 0
				\end{array}\right)}$ & -1\\
			\hline
			3241        & $(134)$      & 1     &
			$\tiny{\left( \begin{array}{ccc}
				-1 & 1 & -1\\
				0 & 1 & 0\\
				1 & 0 & 0
				\end{array}\right)}$ &
			$\tiny{\left( \begin{array}{cc}
				0 & -1\\
				1 & -1\\
				\end{array}\right)}$ &
			$\tiny{\left(\begin{array}{ccc}
				-1 & 0 & 1\\
				-1 & 1 & 0\\
				-1 & 0 & 0
				\end{array}\right)}$ & 1\\
			\hline
			2341        & $(1234)$      & 1     &
			$\tiny{\left( \begin{array}{ccc}
				-1 & 1 & -1\\
				-1 & 0 & 0\\
				0 & -1 & 0
				\end{array}\right)}$ &
			$\tiny{\left( \begin{array}{cc}
				-1 & 1\\
				0 & 1\\
				\end{array}\right)}$ &
			$\tiny{\left(\begin{array}{ccc}
				1 & -1 & 0\\
				1 & 0 & -1\\
				1 & 0 & 0
				\end{array}\right)}$ & -1\\
			\hline
			4321        & $(14)(23)$      & 1     &
			$\tiny{\left( \begin{array}{ccc}
				-1 & 1 & -1\\
				0 & 0 & -1\\
				0 & -1 & 0
				\end{array}\right)}$ &
			$\tiny{\left( \begin{array}{cc}
				1 & 0\\
				0 & 1\\
				\end{array}\right)}$ &
			$\tiny{\left(\begin{array}{ccc}
				0 & 1 & -1\\
				1 & 0 & -1\\
				0 & 0 & -1
				\end{array}\right)}$ & 1\\
			\hline
			3421        & $(1324)$      & 1     &
			$\tiny{\left( \begin{array}{ccc}
				-1 & 1 & -1\\
				0 & 0 & -1\\
				1 & 0 & 0
				\end{array}\right)}$ &
			$\tiny{\left( \begin{array}{cc}
				1 & 0\\
				1 & -1\\
				\end{array}\right)}$ &
			$\tiny{\left(\begin{array}{ccc}
				0 & 1 & -1\\
				-1 & 1 & 0\\
				0 & 1 & 0
				\end{array}\right)}$ & -1\\
			\hline
			2431        & $(124)$      & 1     &
			$\tiny{\left( \begin{array}{ccc}
				-1 & 1 & -1\\
				-1 & 0 & 0\\
				0 & 0 & 1
				\end{array}\right)}$ &
			$\tiny{\left( \begin{array}{cc}
				-1 & 1\\
				-1 & 0\\
				\end{array}\right)}$ &
			$\tiny{\left(\begin{array}{ccc}
				1 & -1 & 0\\
				0 & -1 & 1\\
				0 & -1 & 0
				\end{array}\right)}$ & 1\\
			\hline
			4231        & $(14)$      & 1     &
			$\tiny{\left( \begin{array}{ccc}
				-1 & 1 & -1\\
				0 & 1 & 0\\
				0 & 0 & 1
				\end{array}\right)}$ &
			$\tiny{\left( \begin{array}{cc}
				0 & -1\\
				-1 & 0\\
				\end{array}\right)}$ &
			$\tiny{\left(\begin{array}{ccc}
				-1 & 0 & 1\\
				0 & -1 & 1\\
				0 & 0 & 1
				\end{array}\right)}$ & -1\\
			\hline
			4132        & $(142)$      & 1     &
			$\tiny{\left(\begin{array}{ccc}
				0 & -1 & 0\\
				1 & -1 & 1\\
				0 & 0 & 1
				\end{array}\right)}$ &
			$\tiny{\left( \begin{array}{cc}
				0 & -1\\
				1 & -1\\
				\end{array}\right)}$ &
			$\tiny{\left( \begin{array}{ccc}
				1 & 0 & -1\\
				0 & 0 & -1\\
				0 & 1 & -1
				\end{array}\right)}$ & 1\\
			\hline
			1432        & $(24)$      & 1     &
			$\tiny{\left( \begin{array}{ccc}
				1 & 0 & 0\\
				1 & -1 & 1\\
				0 & 0 & 1
				\end{array}\right)}$ &
			$\tiny{\left( \begin{array}{cc}
				-1 & 1\\
				0 & 1\\
				\end{array}\right)}$ &
			$\tiny{\left(\begin{array}{ccc}
				-1 & 1 & 0\\
				0 & 1 & 0\\
				0 & 1 & -1
				\end{array}\right)}$ & -1\\
			\hline
			3412        & $(13)(24)$      & 1     &
			$\tiny{\left( \begin{array}{ccc}
				0 & 0 & 1\\
				1 & -1 & 1\\
				1 & 0 & 0
				\end{array}\right)}$ &
			$\tiny{\left( \begin{array}{cc}
				1 & 0\\
				0 & 1\\
				\end{array}\right)}$ &
			$\tiny{\left(\begin{array}{ccc}
				0 & -1 & 1\\
				0 & -1 & 0\\
				1 & -1 & 0
				\end{array}\right)}$ & 1\\
			\hline
			4312        & $(1423)$      & 1     &
			$\tiny{\left( \begin{array}{ccc}
				0 & 0 & 1\\
				1 & -1 & 0\\
				0 & -1 & 0
				\end{array}\right)}$ &
			$\tiny{\left( \begin{array}{cc}
				1 & 0\\
				1 & -1\\
				\end{array}\right)}$ &
			$\tiny{\left(\begin{array}{ccc}
				0 & -1 & 1\\
				0 & 0 & 1\\
				-1 & 0 & 1
				\end{array}\right)}$ & -1\\
			\hline
			1342        & $(234)$      & 1     &
			$\tiny{\left( \begin{array}{ccc}
				1 & 0 & 0\\
				1 & -1 & 1\\
				0 & -1 & 0
				\end{array}\right)}$ &
			$\tiny{\left( \begin{array}{cc}
				-1 & 1\\
				-1 & 0\\
				\end{array}\right)}$ &
			$\tiny{\left(\begin{array}{ccc}
				-1 & 1 & 0\\
				-1 & 0 & 0\\
				-1 & 0 & 1
				\end{array}\right)}$ & 1\\
			\hline
			3142        & $(1342)$      & 1     &
			$\tiny{\left( \begin{array}{ccc}
				0 & -1 & 0\\
				1 & -1 & 1\\
				1 & 0 & 0
				\end{array}\right)}$ &
			$\tiny{\left( \begin{array}{cc}
				0 & -1\\
				-1 & 0\\
				\end{array}\right)}$ &
			$\tiny{\left(\begin{array}{ccc}
				1 & 0 & -1\\
				1 & 0 & 0\\
				1 & -1 & 0
				\end{array}\right)}$ & -1\\
			\hline
			2143        & $(12)(34)$      & 1     &
			$\tiny{\left( \begin{array}{ccc}
				0 & -1 & 0\\
				-1 & 0 & 0\\
				-1 & 1 & -1
				\end{array}\right)}$ &
			$\tiny{\left( \begin{array}{cc}
				1 & 0\\
				0 & 1\\
				\end{array}\right)}$ &
			$\tiny{\left(\begin{array}{ccc}
				-1 & 0 & 0\\
				-1 & 0 & 1\\
				-1 & 1 & 0
				\end{array}\right)}$ & 1\\
			\hline
			1243        & $(34)$      & 1     &
			$\tiny{\left( \begin{array}{ccc}
				1 & 0 & 0\\
				0 & 1 & 0\\
				-1 & 1 & -1
				\end{array}\right)}$ &
			$\tiny{\left( \begin{array}{cc}
				1 & 0\\
				1 & -1\\
				\end{array}\right)}$ &
			$\tiny{\left(\begin{array}{ccc}
				1 & 0 & 0\\
				1 & -1 & 0\\
				1 & 0 & -1
				\end{array}\right)}$ & -1\\
			\hline
			4213        & $(143)$      & 1     &
			$\tiny{\left( \begin{array}{ccc}
				0 & 0 & 1\\
				0 & 1 & 0\\
				-1 & 1 & -1
				\end{array}\right)}$ &
			$\tiny{\left( \begin{array}{cc}
				-1 & 1\\
				-1 & 0\\
				\end{array}\right)}$ &
			$\tiny{\left(\begin{array}{ccc}
				0 & 0 & -1\\
				0 & 1 & -1\\
				1 & 0 & -1
				\end{array}\right)}$ & 1\\
			\hline
			2413        & $(1243)$      & 1     &
			$\tiny{\left( \begin{array}{ccc}
				0 & 0 & -1\\
				-1 & 0 & 0\\
				-1 & 1 & -1
				\end{array}\right)}$ &
			$\tiny{\left( \begin{array}{cc}
				0 & -1\\
				-1 & 0\\
				\end{array}\right)}$ &
			$\tiny{\left(\begin{array}{ccc}
				0 & 1 & 0\\
				0 & 1 & -1\\
				-1 & 1 & 0
				\end{array}\right)}$ & -1\\
			\hline
			1423        & $(243)$      & 1     &
			$\tiny{\left( \begin{array}{ccc}
				1 & 0 & 0\\
				0 & 0 & -1\\
				-1 & 1 & -1
				\end{array}\right)}$ &
			$\tiny{\left( \begin{array}{cc}
				0 & -1\\
				1 & -1\\
				\end{array}\right)}$ &
			$\tiny{\left(\begin{array}{ccc}
				0 & -1 & 0\\
				1 & -1 & 0\\
				0 & -1 & 1
				\end{array}\right)}$ & 1\\
			\hline
			4123        & $(1432)$      & 1     &
			$\tiny{\left( \begin{array}{ccc}
				0 & -1 & 0\\
				0 & 0 & -1\\
				-1 & 1 & -1
				\end{array}\right)}$ &
			$\tiny{\left( \begin{array}{cc}
				-1 & 1\\
				0 & 1\\
				\end{array}\right)}$ &
			$\tiny{\left(\begin{array}{ccc}
				0 & 0 & 1\\
				-1 & 0 & 1\\
				0 & -1 & 1
				\end{array}\right)}$ & -1\\
			\hline
		\end{tabular}
		\caption{The irreducible representations of the symmetric group $S_4$.}
		\label{taula:S_4_permutacions}
	\end{center}
\end{table}
This gives:
\begin{enumerate}
\item[$(i)$] $(4=4)$: \\
$H_1= \Stab_{S_4}(\{1,2,3,4\})=S_4$ and so $\rho_i(H_1)=\sum_{g\in S_4}\rho_i(g)$ for $i=1,\ldots,5$.
\item[$(ii)$] $(4=3+1)$: \\
$H_2= \Stab_{S_4}(\{1,2,3\})\cap \Stab_{S_4}(4)=S_3$ and so $\rho_i(H_2)=\sum_{g\in S_3}\rho_i(g)$ for $i=1,\ldots,5$.
\item[$(iii)$] $(4=2+2)$: \\
$H_3= \Stab_{S_4}(\{1,2\})\cap \Stab_{S_4}(\{3,4\})=\{e,(12),(34),(12)(34)\}$ and so $\rho_i(H_3)=\sum_{g\in H_3}\rho_i(g)$ for $i=1,\ldots,5$.
\item[$(iv)$] $(4=2+1+1)$: \\
$H_4= \Stab_{S_4}(\{1,2\})\cap \Stab_{S_4}(3)\cap \Stab_{S_4}(4)=\{e,(12)\}$ and so $\rho_i(H_4)=\rho_i(e)+\rho_i((12))$ for $i=1,\ldots,5$.
\item[$(v)$] $(4=1+1+1+1)$: \\
$H_5= \Stab_{S_4}(1)\cap \Stab_{S_4}(2)\cap \Stab_{S_4}(3)\cap \Stab_{S_4}(4)=\{e\}$ and so $\rho_i(H_5)=\rho_i(e)=d_{\rho_i}$ for $i=1,\ldots,5$.
\end{enumerate}
The ranks for the cases $(i)$-$(v)$, together with the corresponding spectra, are shown in Table \ref{taula:S_4_espectres}. Notice that, in the last but one row of the same table, we have the spectrum of the whole graph $P(4)$ again.

\begin{table}[h!]
	\scriptsize
	\smallskip\noindent
	\begin{center}
		\begin{tabular}{|c|c|c|c|c|c||c|}
			\hline
			& $4$      & $3+1$        & $2+2$     & $2+1+1$  & $1+1+1+1$ & Spectrum\\
			& $\rho_1$ & $\rho_2$     & $\rho_3$  & $\rho_4$ & $\rho_5$  &            \\
			\hline\hline
			rank$(\rho(H_1))$ & 1        & 0            & 0         & 0        & 0 & $\{[3]^1\}$\\
			\hline
			rank$(\rho(H_2))$ & 1        & 1            & 0         & 0        & 0 & $\{[3]^1,[2]^1,[0]^1,$\\
			                  &          &              &           &          &   &
			$[-1]^1\}$ \\
			\hline
			rank$(\rho(H_3))$ & 1        & 1            & 1         & 0        & 0 & $\{[3]^1,[2]^2,[0]^2,$\\
			                  &          &              &           &          &   &
			$[-1]^1\}$ \\
			\hline
			rank$(\rho(H_4))$ & 1        & 2            & 1         & 1        & 0 & $\{[3]^1,[2]^3,[0]^3,$\\
			                  &          &              &           &          &   &
			$[\frac{-1\pm\sqrt{17}}{2}]^1,$ \\
			&          &              &           &          &   &
			$[-1]^2,[-2]^1\}$ \\
			\hline
			rank$(\rho(H_5))$ & 1        & 3            & 2         & 3        & 1 &
			$\{[3]^1,[2]^5,[0]^5,$ \\
			                  &          &              &           &          &   &
			$[\frac{-1\pm\sqrt{17}}{2}]^3,$ \\
			                  &          &              &           &          &   &
			$[0]^5,[-1]^4,[-2]^3\}$ \\
			\hline\hline
			$\spec \rho(B)$   & $\{[3]^1\}$  & $\{[2]^1,[0]^1,[-1]^1\}$ & $\{[2]^1,[0]^1\}$ & $\{[\frac{-1\pm\sqrt{17}}{2}]^1,[-2]^1\}$ & $\{[-1]^1\}$ & \\
			\hline
		\end{tabular}
		\caption{Spectra of the quotient graphs of $P(4)$.}
		\label{taula:S_4_espectres}
	\end{center}
\end{table}
\end{example*}

\begin{example*}[\textbf{New mixed graph} $\Gamma(3,3,2)$]
Consider the mixed graph $\Gamma(3,3,2)$ again.
	Now, by using Theorem \ref{theo-sp}, the spectrum of the whole graph  is the union of the following spectra:
		\begin{enumerate}
			\item[$(i)$]
			$1\cdot \spec \rho_1(B)=\{[3]^1\}$,
			\item[$(ii)$]
			$3\cdot \spec \rho_2(B)=\{ [1]^1, [-1]^1\}$,
			\item[$(iii)$]
			$2\cdot \spec \rho_3(B)=\{[\pm \sqrt{3}]^1\}$,
			\item[$(iv)$]
			$3\cdot \spec \rho_4(B)=\{ [1]^2, [-1]^1\}$,
			\item[$(v)$]
			$1\cdot \spec \rho_5(B)=\{[-3]^1\}$,
		\end{enumerate}
which gives \eqref{eq:espectre-new-mixed}.
Then, the spectra of the different quotient graphs, induced from each partition, are again computed by using the ranks of $\rho_i(H_j)$ for the group stabilizers $H_j$, as in the previous example for $P(4)$. The obtained results are shown in Table \ref{taula:Gamma(3,3,2)_espectres}, where we indicate the spectrum of the whole new mixed graph in the last but one row.
\begin{table}[h!]
	\scriptsize
	\smallskip\noindent
	\begin{center}
		\begin{tabular}{|c|c|c|c|c|c||c|}
			\hline
			& $4$      & $3+1$        & $2+2$     & $2+1+1$  & $1+1+1+1$ & Spectrum\\
			& $\rho_1$ & $\rho_2$     & $\rho_3$  & $\rho_4$ & $\rho_5$  &            \\
			\hline\hline
			rank$(\rho(H_1))$ & 1        & 0            & 0         & 0        & 0 & $\{[3]^1\}$\\
			\hline
			rank$(\rho(H_2))$ & 1        & 1            & 0         & 0        & 0 & $\{[3]^1,[1]^2,[-1]^1\}$\\
			\hline
			rank$(\rho(H_3))$ & 1        & 1            & 1         & 0        & 0 & $\{[3]^1,[\pm\sqrt{3}]^1,[1]^2,[-1]^1\}$\\
			\hline
			rank$(\rho(H_4))$ & 1        & 2            & 1         & 1        & 0 & $\{[3]^1,[\pm\sqrt{3}]^1,[1]^5,[-1]^4\}$\\
			\hline
			rank$(\rho(H_5))$ & 1        & 3            & 2         & 3        & 1 &
			$\{[3]^1,[\pm\sqrt{3}]^2,[1]^9,[-1]^9,[-3]^1\}$ \\
			\hline\hline
			$\spec \rho(B)$   & $\{[3]^1\}$  & $\{[1]^1,[-1]^2\}$ & $\{[\pm\sqrt{3}]^1\}$ & $\{[1]^2,[-1]^1\}$ & $\{[-3]^1\}$ & \\
			\hline
		\end{tabular}
		\caption{Spectra of the quotient graphs of $\Gamma(3,3,2)$.}
		\label{taula:Gamma(3,3,2)_espectres}
	\end{center}
\end{table}
\end{example*}




\begin{thebibliography}{00}
	
\bibitem{ak89}
S. B. Akers and B. Krishnamurthy,
A group-theoretic model for symmetric interconnection networks,
{\emph{IEEE Trans. Comput.}} {\bf 38} (1989), no. 4, 555--566.

\bibitem{b93}
M. Burrow,
{\em Representation Theory of Finite Groups},
Dover, New York, 1993.

\bibitem{cib11}
J. Cibulka,
On average and highest number of flips in pancake sorting,
{\it Theor. Comput. Sci.} {\bf 412} (2011) 822--834.

\bibitem{cf95}
F. Comellas and M. A. Fiol,
Vertex-symmetric digraphs with small diameter,
{\it Discrete Appl. Math.} {\bf 58} (1995), no. 1, 1--12.



\bibitem{Da19}
C. Dalf\'o,
A new general family of mixed graphs,
\textit{Discrete Appl. Math} (2019), in press,
\texttt{doi.org/10.1016/j.dam.2018.12.016}.

\bibitem{dfmrs17}
C. Dalf\'o, M. A. Fiol, M. Miller, J. Ryan,  and J. \v{S}ir\'a\v{n},
An algebraic approach to lifts of digraphs, {\em Discrete Appl. Math.} (2019), in press, \newline \texttt{doi.org/10.1016/j.dam.2018.10.040}.

\bibitem{dfs19}
C. Dalf\'o, M. A. Fiol, and J. \v{S}ir\'a\v{n},
The spectra of lifted digraphs, {\em J. Algebraic Combin.} (2019), in press, {\tt doi.org/10.1007/s10801-018-0862-y}.

\bibitem{dfss19}
C. Dalf\'o, M. A. Fiol, S. Pavl\'ikov\'a, and J. \v{S}ir\'a\v{n},
Spectra and eigenspaces of arbitrary lifts of graphs,
submitted (2019).

\bibitem{d75}
H. Dweighter,
Elementary problems and solutions, problem E2569,
{\it Amer. Math. Monthly} {\bf 82} (1975), no. 10, 1010.


\bibitem{fmc93}
V. Faber, J. W. Moore, and W. Y. C. Chen,
Cycle prefix digraphs for symmetric interconnection networks,
{\it Networks} {\bf 23} (1993) 641--649.

\bibitem{gp79}
W. H. Gates and C. H. Papadimitriou,
Bounds for sorting by prefix reversal,
{\it Discrete Math.} {\bf 27} (1979) 47--57.

\bibitem{g93}
C. D. Godsil,
{\it Algebraic Combinatorics},
Chapman and Hall, New York, 1993.

\bibitem{gt77}
J. L. Gross and T. W. Tucker, Generating all graph coverings by permutation voltage assignments,
{\em Discrete Math.} {\bf 18} (1977) 273--283.

\bibitem{h97}
M. H. Heydari and I. H. Sudborough,
On the diameter of the pancake network,
{\it J. Algorithms} {\bf 25} (1997) 67--94.

\bibitem{JaLi}
G. James and M. Liebeck,
{\em Representations and Characters of Groups},
2nd ed., Cambridge Univ. Press, 2001.

\bibitem{k13}
E. Konstantinova, On some structural properties of star and pancake graphs, in {\em Information Theory, Combinatorics, and Search Theory}, 472--487, Lecture Notes in Comput. Sci. {\bf  7777}, Springer, Heidelberg, 2013.

\bibitem{ms13}
M. Miller and J. \v{S}ir\'a\v{n},
Moore graphs and beyond: A survey of the degree/diameter problem,
{\it Electron. J. Combin.} {\bf 20(2)} (2013) \#DS14v2.



\bibitem{Sl07}
N.~J.~A.~Sloane,
The on-line encyclopedia of integer sequences, A058986,\newline
\texttt{https://oeis.org}.
\end{thebibliography}
\end{document}